\documentclass[11pt,twoside]{amsart}
\usepackage{amssymb}
\usepackage{graphicx,color}

\newtheorem{thm}{Theorem}[section]
\newtheorem{prop}[thm]{Proposition}

\newtheorem{cor}[thm]{Corollary}
\newtheorem{defn}[thm]{Definition}

\newtheorem{ex}[thm]{Example}
\newtheorem{rem}[thm]{Remark}
\newtheorem{conj}[thm]{Conjecture}
\newtheorem{prob}[thm]{Problem}

\newcommand{\skipit}[1]{{}}
\newcommand{\prfend}{\hbox to7pt{\hfil}
\par\vskip-\baselineskip\hbox to\hsize
{\hfil\vbox {\hrule width6pt height6pt}}\vskip\baselineskip}

\newcommand{\ZZ}{\mathbb{Z}}

\newcommand {\PP}{\mathbb{P}}

\DeclareMathOperator{\Image}{Im}
\DeclareMathOperator{\per}{per}

\newcommand{\myarrow}[2]{\hbox to #1pt{\hfil$\to$\hfil}{\hskip-#1pt{\raise
10pt\hbox to#1pt{\hfil$\scriptscriptstyle #2$\hfil}}}}

\begin{document}

\title{On the coefficients of the permanent and the determinant of a circulant matrix. Applications}

\author[Liena Colartes]{Liena Colarte}
\address{Department de matem\`{a}tiques i Inform\`{a}tica, Universitat de Barcelona, Gran Via de les Corts Catalanes 585, 08007 Barcelona,
Spain}
\email{lcolargo8@alumnes.ub.edu}

\author[Emilia Mezzetti]{Emilia Mezzetti}
\address{Dipartimento di Matematica e  Geoscienze, Universit\`a di
Trieste, Via Valerio 12/1, 34127 Trieste, Italy}
\email{mezzette@units.it}

\author[Rosa M. Mir\'o-Roig]{Rosa M. Mir\'o-Roig}
\address{Department de matem\`{a}tiques i Inform\`{a}tica, Universitat de Barcelona, Gran Via de les Corts Catalanes 585, 08007 Barcelona,
Spain}
\email{miro@ub.edu}

\author[Mart\'{\i} Salat]{Mart\'{\i} Salat}
\address{Department de matem\`{a}tiques i Inform\`{a}tica, Universitat de Barcelona, Gran Via de les Corts Catalanes 585, 08007 Barcelona,
Spain}
\email{msalatmo7@alumnes.ub.edu}

\begin{abstract} Let $d(N)$ (resp. $p(N)$) be the number of summands in the determinant (resp. permanent) of an $N\times N$ circulant matrix $A=(a_{ij})$ given by $a_{ij}=X_{i+j}$ where $i+j$ should be considered mod $N$. This short note is  devoted to prove that $d(N)=p(N)$ if and only if $N$ is a prime power. We then give an application to homogeneous monomial ideals failing the Weak Lefschetz property.
\end{abstract}

\thanks{Acknowledgments:   The second author is  member of INdAM - GNSAGA and is supported by PRIN
``Geometry of algebraic varieties''.   The third  author was partially   supported
by  MTM2016--78623-P.
\\ {\it 2010 Mathematic Subject Classification. 15B05, 15A15, 13E10.}}

\maketitle


\markboth{L. Colarte, E. Mezzetti, R. M. Mir\'o-Roig, M. Salat}{Permanent and determinant of a circulant matrix}


\large

\section{Introduction}

We say that a square $N\times N$ matrix $A$ is a {\em circulant} matrix if each row is obtained applying the cycle
$(\alpha _N,\alpha _0,\ldots ,\alpha_{N-1})$ to the preceding row. Circulant matrices have been
studied for a long time and from many different perspectives, for instance in \cite{BN}, \cite{KS},
\cite{LWW}, \cite{Malenf},\cite{T}, \cite{W} and \cite{V}.
In this note, we will focus our attention on the generic circulant matrix $Circ(x_0,x_1,\ldots ,x_{N-1})$ given by the first row
$(x_0 \ x_1\  \ldots  \ x_{N-1})$ (where $x_{0},\ldots, x_{N-1}$ are indeterminates)
and we will address the problem of determining the number of
terms of its determinant $\det(Circ(x_0,\ldots, x_{N-1}))$ (resp. permanent $\per(Circ(x_0,\ldots, x_{N-1}))$). More precisely,
$\det(Circ(x_0,\ldots, x_{N-1}))$ and  $\per(Circ(x_0, \ldots, x_{N-1}))$ are homogeneous polynomials of degree $N$ in
$k[x_0,x_1,\ldots,x_{N-1}]$ and we would like to determine the  number $d(N)$ (resp. $p(N)$) of monomials
appearing in the expansion of $\det(Circ(x_0, \ldots,x_{N-1}))$ (resp.   $\per(Circ(x_0, \ldots,x_{N-1}))$).

In \cite{BN}, it was shown  that the monomials $x_0^{a_0}x_1^{a_1}\cdots x_{N-1}^{a_{N-1}}$ appearing in $\per(Circ(x_0, \ldots,x_{N-1}))$
 are precisely those given by the solutions of the system
 $$
\begin{array}{c}
a _0 + 2a _1 + \cdots  + Na _{N-1} \equiv 0 \pmod{N} \\
 a _0 + \cdots + a _{N-1} = N.
 \end{array}
$$
Unfortunately an analogous result for  the monomials $x_0^{a_0}x_1^{a_1}\cdots x_{N-1}^{a_{N-1}}$ appearing in the expansion
of  $\det(Circ(x_0, \ldots,x_{N-1}))$ is not known. We easily check that $d(N)\le p(N)$.
In addition, in \cite{T}, Thomas proved that $d(N)=p(N)$ for $N$ a power of a prime and asked whether the converse is true. This short note is devoted to prove this result $d(N)=p(N)$ if and only if  $N$ is a
power of a prime.

  In the second part of this note, we generalize the results in \cite{MeMR2} and we prove that the artinian ideal $I^N_{0,1,\ldots ,N-1}\subset k[x_0,x_1,\ldots ,x_{N-1}]$ generated by the monomials appearing in the
  expansion of  $\per(Circ(x_0,\ldots,x_{N-1}))$ is a {\em Togliatti system}, i.e. it fails the {\em Weak Lefschetz property} in degree $N-1$ (see Definition 4.1); even more, it is a {\em GT-system} (see section 4) and we
  analyze whether it is minimal.  In \cite{MeMR2}, the minimality problem of GT-systems in $k[x,y,z]$ was related to the vanishing of the coefficients in the determinant of certain circulant matrices. As an application
  of our results about the number of monomials appearing in   the expansion of the determinant and the permanent of a circulant matrix we are able to conclude that $I^N_{0,1,\ldots ,N-1}$ is a {\em minimal} Togliatti
  system if and only if $N$ is a power of a prime (see Theorem \ref{thm_main2}).

\vskip 4mm
Next we outline the structure of this note.  In
Section~\ref{defs and  prelim results}, we collect the definition and basic results on the determinant (resp. permanent) of a circulant matrix and we collect examples to illustrate that the number of  terms $d(N)$  in the determinant of an $N\times N$  circulant matrix $A$ could be strictly less than the number of  terms  in the permanent  of  $A$. Section \ref{permanent}  contains the main result of this paper, namely, $d(N)=p(N)$ if and only if $N$ is a power of a prime (see Theorem \ref{mainthm}). In the last section, we apply this result to study whether a GT-system is minimal. We prove that,  given an integer $N\ge 3$, the GT-system $I^N_{0,1,\ldots ,N-1}\subset K[x_0,x_1,\ldots  ,x_{N-1}]$  is minimal if and only if $N$ is power of a prime integer (see Theorem \ref{thm_main2}) and we finish our paper with a Conjecture based on our previous results and on
many examples computed with Macaulay2.

\section{Preliminaries}
\label{defs and prelim results}

Throughout this paper, for any $N\times N$ matrix $A=(a_{i,j})$, we denote by $\det(A)$ its determinant and by $\per(A)$ its permanent defined, as usual, by
$$\begin{array}{rcl}\det(A) & = & \sum _{\sigma \in \Sigma _N}(-1)^{\epsilon (\sigma)}\prod a_{i\sigma(i)} \ \ \text{ and} \\ \\
\per(A) & = & \sum _{\sigma \in \Sigma _N}\prod a_{i\sigma(i)}
\end{array}
$$
where the sum  extends over all elements $\sigma $ of the symmetric group $\sum_N$ and $\epsilon (\sigma)$ denotes the signature of the permutation $\sigma $. In this paper, we are interested in computing the non-zero terms of the determinant and of the permanent of a circulant matrix. So, let us start this section by recalling the definition and the basic properties on circulant matrices needed in the sequel. The reader should read \cite{T}, \cite{W} and \cite{Malenf}  for more details.

\begin{defn} \rm An $N\times N$ {\em circulant matrix} is a matrix of the form
$$ Circ(v_0,...,v_{N-1}):=\begin{pmatrix} v_0 & v_1 & \cdots & v_{N-2} & v_{N-1} \\  v_{N-1} & v_0 & \cdots & v_{N-3} & v_{N-2} \\ \vdots & \vdots & \cdots & \vdots & \vdots \\
v_1 & v_ 2 & \cdots & v_{N-1} & v_0
\end{pmatrix}$$ where successive rows are circular permutations of the first row $v_0,\ldots,v_{N-1}$.
\end{defn}

A circulant matrix $ Circ(v_0,\ldots,v_{N-1})$ is a particular form of a Toeplitz
matrix, i.e., a matrix whose elements are constant along the diagonals. A circulant matrix $ Circ(v_0,\ldots,v_{N-1})$ has $N$
eigenvalues, namely, $v_0 + e^pv_1 + e^{2p}v_2 + \cdots + e^{p(N-1)}v_{N-1}$, $0\le p \le N-1$, where $e$ is a primitive $N$-th root of unity. Therefore, it holds:
\begin{equation} \label{det_cir} \begin{array}{rcl} \det(Circ(v_0,\ldots,v_{N-1})) & = &  \left | \begin{matrix} v_0 & v_1 & \cdots & v_{N-2} & v_{N-1} \\  v_{N-1} & v_0 & \cdots & v_{N-3} & v_{N-2} \\ \vdots & \vdots & \cdots & \vdots & \vdots \\
v_1 & v_ 2 & \cdots & v_{N-1}& v_0  \end{matrix} \right | \\ \\
& = & \prod _{j=0}^{N-1} (v_0 + e^jv_1 + e^{2j}v_2 + \cdots + e^{j(N-1)}v_{N-1}). \end{array} \end{equation}
 The computation of the permanent is harder and, in spite of several
 combinatorial interpretations, very little is known so far. The computation of the permanent of a matrix is a challenging problem; it is computationally very hard, even for (0, 1)-matrices. In fact, Valiant
proved that computing the permanent of a (0, 1)-matrix is $\# P$-complete (see \cite{V}). In general, we have:
\begin{equation} \label{per_cir} \begin{array}{rcl} \per(Circ(v_0,\ldots,v_{N-1})) & = & \sum _{\sigma \in \Sigma _N} v_{\sigma(0)} v_{\sigma(1)}\cdots v_{\sigma(N-1)}
\end{array}.
\end{equation}
The product on the right hand side in the equations (\ref{det_cir}) and (\ref{per_cir}), when expanded out, contains ${2N-1}\choose N$ terms and it is still an open problem to find an efficient  formula for the coefficients and decide whether they are zero or not. Some examples of these
determinants (resp. permanents) of generic circulant matrices for small values of $N$ are:

{\small
 $$\begin{array}{rcl} \det (Circ(x,y,z)) &  = & x^3+y^3+z^3-3xyz \\
 \per (Circ(x,y,z)) & = &  x^3+y^3+z^3+3xyz;
  \\  \\
  \det (Circ(x,y,z,t)) & = & x^4-y^4+z^4-t^4-2x^2z^2+2y^2t^2-4x^2yt+4xy^2z-4yz^2t+4xzt^2; \\
 \per (Circ(x,y,z,t)) & = &  x^4+y^4+z^4+t^4+2x^2z^2+2y^2t^2+4x^2yt+4xy^2z+4yz^2t+4xzt^2;
\\  \\
  \det (Circ(x,y,z,t,u)) &  = & x^5+y^5+z^5+t^5+u^5-5x^3yu-5x^3zt-5xy^3z-5y^3tu -5xz^3u \\
  & & -5yz^3t-5xyt^3-5zy^3u-5xtu^3-5yzu^3+5x^2y^2t+5x^2yz^2 +5x^2zu^2 \\
  & &+5x^2t^2u+5xy^2u^2+5xz^2t^2+5y^2z^2u +5y^2tu^2+5yt^2u^2+5z^2tu^2-5xyztu; \\
 \per (Circ(x,y,z,t,u)) & = & x^5+y^5+z^5+t^5+u^5+ 5 t u^3 x + 5 t^2 u x^2 + 5 t^2 u^2 y + 5 t^3 x y + 5 u x^3 y \\
 & & + 5 u^2 x y^2 + 5 t x^2 y^2 + 5 t u y^3  + 5 t^3 u z + 5 u^2 x^2 z + 5 t x^3 z + 5 u^3 y z + 15 t u x y z \\
& &  + 5 t^2 y^2 z + 5 x y^3 z + 5 t u^2 z^2 + 5 t^2 x z^2 + 5 x^2 y z^2 + 5 u y^2 z^2 + 5 u x z^3 + 5 t y z^3;
  \end{array}$$
}
  {\small
 $$\begin{array}{rcl}

 \det(Circ(x,y,z,t,u,v)) & = & x^6-y^6+z^6-t^6+u^6-v^6 +  6 t^4 u z + 6 t^4 v y + 3 t^4 x^2 - 6 t^3 u^2 y - 12 t^3 u v x - \\
 & &
 2 t^3 v^3 - 6 t^3 v z^2 - 12 t^3 x y z - 2 t^3 y^3 +  6 t^2 u^3 x + 9 t^2 u^2 v^2 - 9 t^2 u^2 z^2 + \\
 & &
 18 t^2 u x y^2 + 18 t^2 v^2 x z - 9 t^2 v^2 y^2 - 3 t^2 x^4 + 6 t^2 x z^3 + \\
 & &
 9 t^2 y^2 z^2 - 6 t u^4 v + 12 t u^3 y z - 18 t u^2 x^2 y - 12 t u v^3 z + \\
 & &
 12 t u v x^3 +  12 t u v z^3 - 12 t u y^3 z + 6 t v^4 y - 6 t v^3 x^2 - \\
 & &
 18 t v x^2 z^2 + 6 t v y^4 - 6 t x^2 y^3 + 12 t x^3 y z - 6 t y z^4 - \\
 & &
 6 u^4 x z - 3 u^4 y^2 +  6 u^3 v^2 z + 12 u^3 v x y + 2 u^3 x^3 + 2 u^3 z^3 - 6 u^2 v^3 y - 9 u^2 v^2 x^2 - \\
 & &
 18 u^2 v y z^2 + 9 u^2 x^2 z^2 + 3 u^2 y^4 +  6 u v^4 x + 18 u v^2 y^2 z - \\
 & &
 12 u v x y^3 - 6 u x^4 z + 6 u x^3 y^2 - 6 u x z^4 + 6 u y^2 z^3 +  3 v^4 z^2 - \\
 & &
 12 v^3 x y z - 2 v^3 y^3 + 6 v^2 x^3 z + 9 v^2 x^2 y^2 - 3 v^2 z^4 - 6 v x^4 y + 12 v x y z^3 - 6 v y^3 z^2 + \\
 & &
 2 x^3 z^3 - 9 x^2 y^2 z^2 + 6 x y^4 z\\

 \per(Circ(x,y,z,t,u,v)) & = & x^6+y^6+z^6+t^6+u^6+v^6 +  6 t^4 u z + 6 t^4 v y + 3 t^4 x^2 +  6 t^3 u^2 y + 12 t^3 u v x + \\
 & &
 2 t^3 v^3 + 6 t^3 v z^2 + 12 t^3 x y z + 2 t^3 y^3 +  6 t^2 u^3 x + 9 t^2 u^2 v^2 + 9 t^2 u^2 z^2 + 24 t^2 u v y z + \\
 & &
 12 t^2 u x^2 z + 18 t^2 u x y^2 + 18 t^2 v^2 x z + 9 t^2 v^2 y^2 +  12 t^2 v x^2 y + 3 t^2 x^4 + 6 t^2 x z^3 + \\
 & &
 9 t^2 y^2 z^2 +  6 t u^4 v + 12 t u^3 y z + 24 t u^2 v x z + 12 t u^2 v y^2 + 18 t u^2 x^2 y +  12 t u v^3 z + \\
 & &
 24 t u v^2 x y + 12 t u v x^3 +  12 t u v z^3 + 24 t u x y z^2 + 12 t u y^3 z+  6 t v^4 y + 6 t v^3 x^2 + \\
 & &
 12 t v^2 y z^2 + 18 t v x^2 z^2 + 24 t v x y^2 z + 6 t v y^4 + 6 t x^2 y^3 + 12 t x^3 y z + 6 t y z^4 + \\
 & &
 6 u^4 x z + 3 u^4 y^2 +  6 u^3 v^2 z + 12 u^3 v x y + 2 u^3 x^3 + 2 u^3 z^3 +  6 u^2 v^3 y + 9 u^2 v^2 x^2 + \\
 & &
 18 u^2 v y z^2 + 9 u^2 x^2 z^2 + 12 u^2 x y^2 z + 3 u^2 y^4 +  6 u v^4 x + 12 u v^2 x z^2 + 18 u v^2 y^2 z + \\
 & &
 24 u v x^2 y z + 12 u v x y^3 +  6 u x^4 z + 6 u x^3 y^2 + 6 u x z^4 + 6 u y^2 z^3 +  3 v^4 z^2 + \\
 & &
 12 v^3 x y z + 2 v^3 y^3 +  6 v^2 x^3 z + 9 v^2 x^2 y^2 + 3 v^2 z^4 +  6 v x^4 y + 12 v x y z^3 + 6 v y^3 z^2 + \\
 & &
 2 x^3 z^3 + 9 x^2 y^2 z^2 + 6 x y^4 z

  \end{array}$$
}

As we have seen in these examples when  we expand  $\det(Circ(x_0,\ldots,x_{N-1}))$ we obtain a polynomial in the $x_i$ and
we  define $d(N)$ to be the
number of degree $N$ monomials in this polynomial after like terms have been combined.
So, $d(3)=4$, $d(4)=10$, $d(5)=26 $, $d(6)=68$,  etc. Similarly,
we define $p(N)$ to be the number of terms in  the permanent,
$\per(Circ(x_0,\ldots,x_{N-1}))$, of $Circ(x_0,\ldots,x_{N-1})$. So, $p(3)=4$, $p(4)=10$, $p(5)=26 $, $p(6)=80$,  etc.


\section{The permanent and the determinant of a circulant matrix }
\label{permanent}

In this section we consider the determinant and the permanent of generic matrices, i.e. the entries are indeterminates.

\begin{prob} \label{pblm} To determine the integers $N\ge 1$ such that $d(N)=p(N).$
\end{prob}

From the definition, it is clear that $$d(N) \le p(N)$$ since every term which appears
in  $\det(Circ(x_0,\ldots,x_{N-1}))$ also appears in $\per(Circ(x_0,\ldots,x_{N-1}))$. However, due to cancellations,
some terms appearing in $\per(Circ(x_0,\ldots,x_{N-1}))$
 could be absent in $\det(Circ(x_0,\ldots,x_{N-1}))$, i.e it could be $d(N)
\lneq
 p(N)$ (for example, $d(6)=68<80=p(6)$). To analyze whether $d(N)=p(N)$
  we would like to have an efficient  formula for the coefficients and decide whether they are zero or not.
Let us start determining the non-zero coefficients of $\per(Circ(x_0,\ldots,x_{N-1}))$ and $p(N)$. The function $p(N)$ was
studied in \cite{BN} by Brualdi and Newman, they showed  that $p(N)$ coincides with the number of
solutions to
\begin{equation}\label{sol_p(n)}
\begin{array}{c}
\alpha _0 + 2\alpha _1 + \cdots  + N\alpha _{N-1} \equiv 0 \pmod{N} \\
 \alpha _0 + \cdots + \alpha _{N-1} = N
 \end{array}
\end{equation}
in non-negative integers. They also proved by a generating
function argument that
\begin{equation}\label{p(N)} p(N) =\frac{1}{N}\sum _{k|N }\phi(\frac{N}{k}){2k-1\choose k}
\end{equation}
where $\phi (n)$ is  the Euler's function that counts the
positive integers up to $n$ that are relatively prime to $d$.
Unfortunately, a formula for the coefficients in the expansion of determinant, $\det(Circ(x_0,\ldots,x_{N-1}))$, is not known; a criterium to decide whether a coefficient is non-zero is not available
 and, hence, the value of $d(N)$ is out of reach
despite the fact that its definition seems
at least as natural.
Let us now summarize what is known about the coefficients in the left hand side of the equation (\ref{det_cir}).  To this end we express the determinant of an $N\times N$ circulant matrix as follows:

$$ \det(Circ(x_0,\ldots,x_{N-1}))  = \sum _{0\le a_0\le \cdots \le a_{N-1}\le N-1}c_{a_0\cdots a_{N-1}}x_{a_0}\cdots x_{a_{N-1}}.$$

This sum can also be written as
$$ \det(Circ(x_0,\ldots,x_{N-1}))  = \sum _{0\le M_0,\ldots ,M _{N-1}\le N-1}d_{M_0\cdots M_{N-1}}x_{0}^{M_0}\cdots x_{{N-1}}^{M_{n-1}}$$
where $d_{M_0\cdots M_{N-1}}=c_{a_0\cdots a_{N-1}}$, if $M_0+\cdots +M_{N-1}=N$ and $M_i$ is the multiplicity of $i$,
the number of times the integer $i$ occurs in the index set $[a_0,\ldots ,a_{N-1}]$.

\begin{prop}\label{coefI} With the above notation, if $\alpha_{0}+2\alpha_{1}+\dotsb+N\alpha_{N-1}\equiv 0\pmod{N}$, then:
\begin{itemize}
\item[(1)] If $a_0+a_1+\cdots +a_{N-1}\not\equiv 0 \pmod{N}$, then $c_{a_0\cdots a_{N-1}}=0 $.
\item[(2)] If $N$ is prime and $a_0+a_1+\cdots +a_{N-1}\equiv 0 \pmod{N}$, then $c_{a_0\cdots a_{N-1}}\ne 0 $.
\end{itemize}
\end{prop}

\begin{proof} (1) See \cite{Malenf}; Theorem 1 or \cite{W}; Proposition 10.4.3.

(2) See \cite{Malenf}; Corollary 4 or \cite{W}; Chapter 11.

\end{proof}

\begin{rem}\label{rem_circ} \rm  (1) It is worthwhile to point out that expressed in terms of the set of multiplicities, $[M] = [M_0,M_1, \ldots ,M_{N-1}]$,
$0 \le  M_0,M_1, \ldots  ,M_{N-1} \le  N$, the condition $a_0+a_1+\cdots +a_{N-1}\equiv 0 \pmod{N}$ becomes
$0 M_0 + 1 M_1 +\cdots  + (N - 1) M_{N-1} \equiv 0 \pmod{N}$ subject to the restriction
$M_0 +M_1 +\cdots  +M_{N-1} = N$.

 (2)  Proposition \ref{coefI} (2) is not true if $N$ is not prime. Indeed,
 for $N=6$ we have seen that $c_{0,0,1,3,3,5}$, $c_{0,0,1,2,4,5}$,
 $c_{0,0,2,3,3,4}$, $c_{0,1,1,2,4,4}$, $c_{0,1,1,2,3,5}$, $c_{0,1,2,2,3,4}$,
 $c_{0,1,3,4,4,5}$, $c_{0,2,3,4,4,5}$, $c_{0,2,2,4,5,5}$, $c_{1,2,2,3,5,5}$,
 $c_{1,1,3,4,4,5}$
 $c_{1,2,3,3,4,5}$ are zero in the determinant expansion, but
  they satisfy the previous conditions and  they appear as non-zero coefficients
 in the permanent expansion. Also, for $N=10$ we have $c_{0,0,0,0,1,1,1,3,6,8}=0$ (see, for instance, \cite{W}; pag. 123). More generally, we have
\end{rem}

\begin{prop} \label{key} For $N=M_0+M_1+3$ with $M_0,M_1\ge 1$, the coefficient $c_{0\cdots 01\cdots 1a_{N-3}a_{N-2}a_{N-1}}$ with $M_1+a_{N-3}+a_{N-2}+a_{N-1}\equiv 0 \pmod{N}$ is zero if $N$ divides $(M_1+2)(M_1+1)$ and either
\begin{itemize}
\item $a_{N-3}\le a_{N-2}<N-M_1$, $a_{N-3}+a_{N-2}=N+1-\frac{(M_1+2)(M_1+1)}{N}$ and $a_{N-1}=M_0+2+
\frac{(M_1+2)(M_1+1)}{N}$, or
\item $N-M_1\le a_{N-2}\le a_{N-3}$, $a_{N-2}+a_{N-1}=N+1+\frac{(M_0+2)(M_0+1)}{N}$ and $a_{N-3}=M_0+2-
\frac{(M_0+2)(M_0+1)}{N}$.
\end{itemize}
\end{prop}

\begin{proof} See \cite{Malenf}; Corollary 6.
\end{proof}

Using the theory of symmetric functions H. Thomas proved in \cite{T} that for any prime integer $p\ge 1$ and for any
integer $n\ge 1$ it holds: $d(p^n)=p(p^n)$ but he left open Problem \ref{pblm}. In section 2
and Remark~\ref{rem_circ} we have seen  examples
proving that equality is not always true (for instance, $d(6)<p(6)$).
 We are now ready to state the main result of this paper and explicitly determine when
 $d(N)=p(N)$. Indeed, we have

\begin{thm} \label{mainthm}
Fix $N\ge 1$, then $d(N)=p(N)$ if and only if $N=p^r$, with $p$ a prime integer.
\end{thm}

\begin{proof} If $N=p^r$ with $p$ a prime integer and $r\ge 1$ then by \cite{T} we have $d(p^r)=p(p^r)$. Let us prove the converse. Write $N=nm$ with $1<n<m$ and  $gcd(m,n)=1$. In order to prove that $p(N)\lneq d(N)$
it is enough to exhibit an $N$-tuple $(a_0,a_1,\ldots ,a_{N-1})$ verifying the equations (\ref{sol_p(n)}), i.e. $a _0 + 2a _1 +\cdots  + Na _{N-1} \equiv 0 \pmod{N}$, $a _0 + \cdots + a _{N-1} = N$ and such that the coefficient $c_{a_0\cdots a_{N-1}}=0$. To this end, we apply Bezout's theorem and we write $\lambda m = 1 + \mu n$ with $1\le \lambda, \mu $ and $\lambda m \le N$.

We define $$ M_1:=\mu n -1.$$
We first observe that $N=nm$ divides $(M_1 +1)(M_1 +2)$ (Indeed, $(M_1 +1)(M_1 +2)=\mu n (\mu n +1)=(\mu n)(\lambda m)=\lambda \mu N$). Let us check that $M_1\le N-4$. Notice that $M_1=\mu n-1=\lambda m -2\le N-2$. Therefore, we only need to prove that the case $M_1=N-3$ is not possible. If $M_1=N-3$ then $\mu n +1=M_1+2=N-1$, i.e. $\mu n=N-2=nm-2$. So, $n$ divides 2. Since $1<n$, we get $n=2$ and $\mu = m-1$. Using the equalities $\lambda m=1+\mu n=1+2(m-1)=2m-1$ we obtain that $m$ divides 1 which is a contradiction.

Let us now  define $$\begin{array}{rcl}M_0 &: = & nm-\mu n -2, \\
 A_2 & := & nm-\mu n, \\
A_1 & := & \mu n-\mu \lambda+1, \text{ and } \\
A_3 & := & nm-\mu n+\lambda \mu.
\end{array}
$$

We easily check that
\begin{itemize}
\item $1\le M_0=N-M_1-3,$
\item $M_1+A_1+A_2+A_3\equiv 0 \pmod{N}$, and
\item $A_1\le A_2<N-M_1$, $A_1+A_2=N+1-\frac{(M_1+1)(M_1+2)}{N}$ and $A_3=M_0+2+\frac{(M_1+1)(M_1+2)}{N}$.
\end{itemize}
Therefore, we can apply Proposition \ref{key} and conclude that $c_{0\cdots 01\cdots 1A_{1}A_{2}A_{3}}=0$.
\end{proof}

As an immediate consequence of Theorem \ref{mainthm} we can  slightly generalize Proposition \ref{coefI} and we get:

\begin{cor} \label{improve-coefI} With the notation of Proposition \ref{coefI}, it holds:
 If $N$ is a power of a prime and $a_0+a_1+\cdots +a_{N-1}\equiv 0 \pmod{N}$, then $c_{a_0\cdots a_{N-1}}\ne 0 $.
\end{cor}



\section{Galois-Togliatti systems and Galois covers}
\label{GT-systems}

In this section, we will apply the above result to study the minimality of certain Galois-Togliatti systems (GT-systems, for short). So, let us start this section recalling the notion of GT-systems and relating  their minimality with the problem of whether $d(N)=p(N)$. To this end, we fix $k$ an algebraically closed field and we set $R:=k[x_0,\ldots ,x_n] $

\begin{defn}\label{wlp}\rm
Let $I\subset R $ be a homogeneous artinian ideal. We say that   $I$ has Weak Lefschetz Property (WLP)
 if there is a  $L \in [R/I]_1$ such that, for all
integers $j$, the multiplication map
\[
\times L: [R/I]_{j-1} \to [R/I]_j
\]
has maximal rank.
\end{defn}

To establish whether an ideal $I\subset R$ has the WLP is a  difficult and  challenging problem
 and even in simple
cases, such as complete intersections, much remains unknown about the presence of the WLP.
 Recently the failure of the WLP has been connected to a large number of problems,
that appear to be unrelated at first glance.
For example, in \cite{MMO}, Mezzetti, Mir\'{o}-Roig and Ottaviani proved that the failure of the WLP is related to the existence of varieties satisfying at least one Laplace equation of order greater than 2 and they proved:

\begin{thm}\label{tea} Let $I\subset R$ be an artinian
ideal
generated
by $r$ homogeneous polynomials $F_1,\ldots,F_{r}$ of degree $d$ and let $I^{-1}$ be its Macaulay inverse system.
If
$r\le {n+d-1\choose n-1}$, then
  the following conditions are equivalent:
\begin{itemize}
\item[(1)] the ideal $I$ fails the WLP in degree $d-1$;
\item[(2)] the  homogeneous forms $F_1,\ldots,F_{r}$ become
$k$-linearly dependent on a general hyperplane $H$ of $\PP^n$;
\item[(3)] the $n$-dimensional   variety
 $X=\overline{\Image (
\varphi )}$
where
$\varphi :\PP^n \dashrightarrow \PP^{{n+d
\choose d}-r-1}$ is the  rational map associated to $(I^{-1})_d$,
  satisfies at least one Laplace equation of order
$d-1$.
\end{itemize}
\end{thm}

\begin{proof} See \cite[Theorem 3.2]{MMO}.
\end{proof}

The above result motivated the following definitions:

\begin{defn}\rm Let $I \subset R$ be an artinian ideal generated by $r$ forms  of degree $d$, and $r \leq {n+d-1\choose n-1}$. We will say:
\begin{itemize}
\item[(i)] $I$ is a \emph{Togliatti system} if it satisfies one of three equivalent conditions in Theorem \ref{tea}.
\item[(ii)] $I$ is a \emph{monomial Togliatti system} if, in addition, $I$ can be generated
by monomials.
\item[(iii)] $I$ is a \emph{smooth Togliatti system} if, in addition, the rational variety $X$ is smooth.
\item[(iv)] A monomial Togliatti system $I$ is \emph{minimal} if  there is no proper subset of the set of generators defining a monomial Togliatti system.

\end{itemize}
\end{defn}

These definitions were introduced in \cite{MeMR} and the names are in honor of Eugenio Togliatti who proved that for $n = 2$ the only smooth
Togliatti system of cubics is $$I = (x_0^3,x_1^3,x_2^3,x_0x_1x_2)\subset k[x_0,x_1,x_2]$$ (see \cite{BK}, \cite{T1}[ and \cite{T2}). The systematic study of Togliatti systems was initiated in \cite{MMO} and for recent results the reader can see \cite{MeMR}, \cite{MeMR2}, \cite{MRM}, \cite{AMRV} and \cite{MS}.

In this paper, we will restrict our attention to a particular case of Togliatti systems, the so-called GT-systems introduced, for the case $N=3$, in \cite{MeMR2}. To define them we need to fix some extra notation. Fix $N\ge 3$, $d\ge 3$ and $e$ a primitive  $d$-th root of the unity. Since any representation of $\ZZ/d\ZZ$ in $GL(N,\ZZ)$ can be diagonalized, we can assume that it is represented by a matrix of the form
$$M:=M(\alpha_0,\alpha_1,\ldots,\alpha_{N-1})=\begin{pmatrix} e^{\alpha _0} & 0 & \cdots & 0 \\ 0 & e^{\alpha _1} &  \cdots & 0 \\
 &  & \cdots &  \\ 0 & 0 &\cdots & e^{\alpha _{N-1}}
\end{pmatrix}$$
with $gcd(\alpha_0,\alpha_1,\ldots,\alpha_{N-1},d)=1$.

\begin{defn}\rm Fix integers $3\le d\in \ZZ$ and $3\le N\in \ZZ$, $e$ a primitive $d$-th root of 1 and $ M(\alpha_0,\alpha_1,\ldots,\alpha_{N-1})$ a representation of $\ZZ/d\ZZ$ in $GL(N,\ZZ)$. A {\em GT-system} will be an ideal $$I^d_{\alpha_0,\ldots,\alpha _{N-1}}\subset k[x_0,x_1,\ldots,x_{N-1}]$$ generated by all forms of degree $d$ invariant under the action of a  $ M(\alpha_0,\alpha_1,\ldots,\alpha_{N-1})$.
\end{defn}

\begin{ex} \rm  Fix an odd integer
$3\le d=2k+1$. The  monomial artinian ideal
$I=(x_0^d,x_1^d, x_2^d,x_0x_1^{d-2}x_2 , x_0^2x_1^{d-4}x_2^2,\ldots ,
x_0^kx_1x_2^k)\subset K[x_0,x_1,x_2]$ defines a monomial GT-system. Indeed, let $e$ be a primitive root of order $d$ of  1 and consider the representation of the cyclic Galois group $\ZZ/d\ZZ$
 on $GL(3,\ZZ)$ given by the diagonal matrix $M(e^0, e^1, e^2)$. It is easy to check that $I$ is generated by all forms of degree $d$ invariant under the action of $M(e^0, e^1, e^2)$. Therefore, $I$ is a GT-system (Indeed, $I=I^d_{0,1,2}$).
Note that  for $d=3$ we recover the  smooth
Togliatti system of cubics.
\end{ex}

We easily check that a GT-system $I:=I^d_{\alpha_0,\ldots,\alpha _{N-1}}$ is always an artinian ideal since  $x_i^d\in I$ for $i=0,\ldots , N-1$. So, it defines a regular map
$$\psi _I:\PP^{N-1}\longrightarrow \PP^{\mu (I)}$$
where $\mu(I)$ denotes the minimal number of  generators of $I$. The morphism $\psi _I$ is a Galois cover of degree $d$ of the $N-1$ dimensional rational variety $\psi_I(\PP^{N-1})$ with cyclic Galois group $\ZZ/d\ZZ$ represented by the matrix $M(\alpha_0,\alpha_1,\ldots,\alpha_{N-1})$. Generalizing \cite{MeMR2}, Theorem 3.1 and  Proposition 3.2  from $N=2$ to arbitrary $N\ge 2$, we get:

\begin{prop}\label{prop:GTTogl} Fix integers $N\geq 3$ and  $d\geq N $, $e$ a primitive $d$-th root of 1 and
$M(\alpha_0,\alpha_1,\ldots,\alpha_{N-1})$
a representation of $\ZZ/d\ZZ$. Let $I$ denote  the ideal $I^d_{\alpha_0,\ldots,\alpha _{N-1}}\subset K[x_0,x_1,\ldots ,x_{N-1}]$ generated
by all forms of degree $d$ invariant
under the action of $M(\alpha_0,\alpha_1,\ldots,\alpha_{N-1})$. Assume $\mu(I)\le {N+d-2\choose N-2}$.  Then $I$
is a monomial artinian ideal which fails WLP in degree $d-1$ (i.e GT-systems are Togliatti systems).
\end{prop}

\begin{proof} To check that $I:=I^d_{\alpha_0,\ldots,\alpha _{N-1}}$ is a monomial ideal is a straightforward computation that is left to the reader. Let us check that $I$ fails WLP from degree $d-1$ to degree $d$.
Since $\mu(I)\le {N+d-1\choose N-1}$, we have
$$\begin{array}{rcl} \dim R_{d-1}& = & \dim [R/I]_{d-1}={N+d-2\choose N-1} \\
& = & {N+d-1\choose N-1}-{N+d-2\choose N-2} \\
& \le & \dim R_d-\dim \mu (I)= \dim[R/I]_{d}.\end{array}
 $$
So, to see that $I$ fails WLP from degree $d-1$ to degree $d$, we have to show that for any linear form $\ell \in R$ the induced map $\times \ell :[R/I]_{d-1}\longrightarrow [R/I]_d$ is not injective.
By \cite{MMN2}, Proposition 2.2, it is enough to check it for $\ell = x_0+x_1+\cdots +x_{N-1}$.
This is equivalent to prove that   there exists a form
$F_{d-1}\in R$ of degree $d-1$ such that $(x_0+x_1+\cdots +x_{N-1})\cdot F_{d-1}\in I$.
Consider $F_{d-1}=(e^{\alpha _0}x_0+e^{\alpha _1}x_1+\cdots
+e^{\alpha _{N-1}}x_{N-1})(e^{2\alpha _0}x_0+e^{2\alpha _1}x_1+\cdots +e^{2\alpha _{N-1}}x_{N-1})
\cdots
(e^{(d-1)\alpha _0}x_0+e^{(d-1)\alpha _1}x_1+\cdots +e^{(d-1)\alpha _{N-1}}x_{N-1}).$
The homogeneous form of degree $d$, $F=(x_0+x_1+\cdots +x_{N-1})\cdot F_{d-1}$ is invariant under the action of
$M(\alpha_0,\alpha_1,\ldots,\alpha_{N-1})$. Hence, it belongs to $I$ which  proves our result.
\end{proof}

\begin{rem}  \rm The hypothesis $\mu(I)\le {N+d-2\choose d-2}$ in Proposition 4.6 is often satisfied and can be dropped.  For instance we will see in Theorem \ref{thm_main2} that it is verified  when $d=N$ and $(\alpha_0,\alpha_1,\ldots,\alpha _{N-1})=(0,1,\ldots,N-1)$. 
\end{rem}

To determine the minimality of a GT-system $I^d_{\alpha_0,\ldots,\alpha _{N-1}}$ is  a subtle  problem that for $N=3$ was related in \cite{MeMR2} to the determinant of certain circulant matrices. This relation works for arbitrary $N\ge 3$. We quickly recall/generalize it (from $N=3$ to $N\ge 3$). Finally, as application of this relationship and of the results obtained in the previous section,
we will be able to prove or disprove  the minimality of $I^N_{0,1,\ldots,N-1}$.  Proving the minimality of the GT-system $I^d_{\alpha_0,\ldots,\alpha _{N-1}}$ is equivalent to proving that the monomials of degree $d$ invariant under the action of $M(\alpha_0,\alpha_1,\ldots,\alpha_{N-1})$ all appear with non-zero coefficient in the development of the product of linear forms
$(x_0+x_1+\cdots +x_{N-1})(e^{\alpha _0}x_0+e^{\alpha _1}x_1+\cdots
+e^{\alpha _{N-1}}x_{N-1})(e^{2\alpha _0}x_0+e^{2\alpha _1}x_1+\cdots +e^{2\alpha _{N-1}}x_{N-1})
\cdots
(e^{(d-1)\alpha _0}x_0+e^{(d-1)\alpha _1}x_1+\cdots +e^{(d-1)\alpha _{N-1}}x_{N-1}).$

For any integer $d\geq N$ and $0\le \alpha _0<\alpha _1<\cdots <\alpha _{N-1}\leq d$, we consider the $d\times d$ circulant matrix
\begin{equation*}A^d_{\alpha _0,\alpha _1,\ldots ,\alpha _{N-1}}=Circ(0,\ldots, 0,x_0,0,\ldots , 0, x_1,0, \ldots ,0,x_i,0, \ldots ,0, x_{N-1},0,\ldots,0)\end{equation*}
where  $x_i$ is in the position of index $\alpha _i$. According to (\ref{det_cir}) we have $$\det (A^d_{\alpha _0,\alpha _1,\ldots ,\alpha _{N-1}})= \prod _{j=0}^{d-1} (e^{j\alpha _0}x_0 + e^{j\alpha _1}x_1 +\cdots + e^{j\alpha _{N-1}}x_{N-1}).$$
 The determinant of $A^d_{\alpha _0,\alpha _1,\ldots ,\alpha _{N-1}}$ is therefore exactly  the product we are interested in and
we want to prove that all monomials of degree $d$ invariant under the action of $$M(\alpha_0,\alpha_1,\cdots,\alpha_{N-1})$$ appear with non-zero coefficient in $\det(A^d_{\alpha _0,\alpha _1,\ldots ,\alpha _{N-1}})$.

\begin{thm}\label{thm_main2} Fix an integer $N\ge 3$. The GT-system $I^N_{0,1,\ldots,N-1}\subset K[x_0,x_1,\ldots ,x_{N-1}]$  is a monomial Togliatti system. It is minimal if and only if $N$ is power of a prime integer.
\end{thm}

\begin{proof} 
In view of Proposition \ref{prop:GTTogl}, to prove that $I^N_{0,1,\ldots,N-1}$ is a Togliatti system we have to check that 
$\mu(I^N_{0,1,\ldots,N-1})\le {2N-2\choose N-2}$. But $\mu(I^N_{0,1,\ldots,N-1})\leq d(N)\leq p(N)$, so by (\ref{p(N)}) 
it is enough to check that  $$ \frac{1}{N} \sum_{k|N} \varphi(\frac{N}{k}) \binom{2k-1}{k} \leq \binom{2N-2}{N}.$$
 
 First we assume that $N$ is not prime and we consider its prime decomposition $N = q_{1}^{r_{1}} \cdots q_{s}^{r_{s}}$  with $q_{1} < q_{2} < \cdots < q_{s}$. So,
$N' = \frac{N}{q_{1}}$ is the greatest divisor of $N$ different from $N$, $N \geq 2N'$ and $N' \geq 2$.
Since $\varphi(1) = 1$ and $\sum_{k|N} \varphi(\frac{N}{k}) = N$, we can write:
$$\begin{array}{rcl} \frac{1}{N} \sum_{k|N} \varphi(\frac{N}{k}) \binom{2k-1}{k} & = & \frac{1}{N}(\sum_{k|N,k \neq N} \varphi(\frac{N}{k})\binom{2k-1}{k}
+ \binom{2N-1}{N})\leq\\    & \leq &
 \frac{N-1}{N}\binom{2N'-1}{N'} + \frac{1}{N}\binom{2N-1}{N}= \\ \\
&  = &  \frac{(2N'-1)(N-1)}{N(N'-1)} \binom{2N'-2}{N'} + \frac{2N-1}{N(N-1)}\binom{2N-2}{N}.
\end{array}
$$
Therefore,  $$\frac{1}{N} \sum_{k|N} \varphi(\frac{N}{k}) \binom{2k-1}{k}\leq \binom{2N-2}{N} $$
which is equivalent to
$$  \frac{2N'-1}{N'-1} \binom{2N'-2}{N'} \leq \frac{N^{2} - 3N + 1}{(N-1)^{2}} \binom{2N-2}{N}.$$
Using that $\frac{2N'-1}{N'-1} = 1 + \frac{N'}{N'-1} \leq 3$ for all $N' \geq 2$, it is enough to see the following inequality:
$$
\frac{3(N^{2} - 2N + 1)}{N^{2} - 3N + 1} \leq \frac{\binom{2N-2}{N}}{\binom{2N'-2}{N'}}.$$
Since $\frac{3(N^{2} - 2N + 1)}{N^{2} - 3N + 1} \leq 6$ for $N \geq 4$, it only remains to check that $\binom{2N-2}{N}\ge 6\binom{2N'-2}{N'}$. By the Pascal's rule $\binom{n+1}{k} = \binom{n}{k} + \binom{n}{k-1}$, it suffices to verify that $\binom{4N'-2}{2N'}\ge 6\binom{2N'-2}{N'}$.  Applying consecutively the Pascal's rule,
we obtain:
$$\binom{4N'-2}{2N'} = \binom{4N'-3}{2N'} + \binom{4N'-3}{2N'-1} = \cdots = \binom{4N'-5}{2N'} + 3\binom{4N'-5}{2N'-1} + 4\binom{4N'-5}{2N'-2}.$$
Since $2N'-3 \geq N'-2$, we have $\binom{4N'-5}{2N'-2} = \binom{2N'-2 + 2N'- 3}{N'+ N'-2} \geq \binom{2N'-2 + N'-1}{N'} \geq \binom{2N'-2}{N'}$. Similarly,
$\binom{2N'-2 + 2N' - 3}{N' + N' - 1} \geq \binom{2N'-2 + N'-2}{N'} \geq \binom{2N'-2}{N'}$. Summing up, $\binom{4N'-2}{2N'} \geq 7\binom{2N'-2}{N'} + \binom{4N'-5}{2N'}$ which implies
$\binom{4N'-2}{2N'}\ge 6\binom{2N'-2}{N'}$.

\vskip 2mm
If $N$ is prime,  $\frac{1}{N} \sum_{k|N} \varphi(\frac{N}{k}) \binom{2k-1}{k}  = \frac{N-1}{N} + \frac{1}{N}\binom{2N-1}{2N}
= \frac{N-1}{N} + \frac{2N-1}{N(N-1)} \binom{2N-2}{2N}$. In this case, the expected inequality becomes
$\frac{N^{2} - 2N + 1}{N^{2} - 3N + 1} \leq \binom{2N - 2}{N}$, which can be easily reduced to verify that
$\binom{2N-2}{N} \geq 2$, since $\frac{N^{2} - 2N + 1}{N^{2} - 3N + 1} \leq 2$ for any $N \geq 3$. The result follows directly from
the growth of the binomial coefficients, or simply observing the Pascal's triangle.

It remains to prove that $I^N_{0,1,\ldots,N-1}$ is a minimal Togliatti system. First of all we observe that a monomial $m=x_0^{i_0}x_1^{i_1}\cdots x_{N-1}^{i_{N-1}}$ with $i_j\ge 0$ and $\sum _{j=0}^{N-1}i_j=N$ belongs to  $I^N_{0,1,\ldots,N-1}$ if and only if $i_0+2i_1+3i_2+\cdots +Ni_{N-1}\equiv 0 \pmod{N}$. Therefore, the number of generators of $I^N_{0,1,\ldots,N-1}$ is equal to $d(N) =\frac{1}{N}\sum _{d|N }\phi(\frac{N}{d}){2N-1\choose N}$ and the GT-system $I^N_{0,1,\ldots,N-1}$  will be minimal if $d(N)$ coincides
with
the number of non-zero coefficients in the development of the product of linear forms
$(x_0+x_1+\cdots +x_{N-1})
(x_0+ex_1+\cdots +e^{N-1}x_{N-1})
(x_0+e^{2}x_1+\cdots +e^{2(N-1)}x_{N-1})
\cdots
(x_0+e^{d-1}x_1+\cdots +e^{(d-1)(N-1)}x_{N-1}).$ In other words, $I^N_{0,1,\ldots,N-1}$ will be minimal if and only if $d(N)=p(N)$ and, by Theorem \ref{mainthm}, if and only if $N$ is a power of a prime integer.
\end{proof}

\begin{prop} Fix  integers $N\ge 3$ and $d\ge N$.  If $d$ is a power of a prime then the GT-system $I^d_{0,1,\ldots,N-1}\subset K[x_0,x_1,\ldots ,x_{N-1}]$  is minimal.
\end{prop}
\begin{proof} Indeed, if the coefficient of an invariant monomial is zero in $$I^d_{0,1,\ldots,N-1}\subset K[x_0,x_1,\ldots ,x_{N-1}]$$ it is zero also in the ideal in $d$ variables hence $d$ is not a power of prime.
\end{proof}

We end this note with a conjecture based on Theorem \ref{thm_main2}, on our results in \cite{MeMR2}, and on many examples computed with Macaulay2 \cite{M}.

\begin{conj} \rm
Fix an integer $d\ge 3$ and integers $1\le n<m\le d-1$ such that $gcd(n,m,d)=1$. Then, the GT-system $I^d_{0,n,m}\subset k[x,y,z]$ is minimal.
\end{conj}



\end{document}